\documentclass[11pt]{article}

\usepackage{amsmath,amssymb,amsthm}
\usepackage[T1]{fontenc}
\usepackage[utf8]{inputenc}
\usepackage[english]{babel}

\usepackage{xcolor}
\usepackage{soul}

\newtheorem{proposition}{Proposition}[section]
\newtheorem{theo}[proposition]{Theorem}
\newtheorem{lem}[proposition]{Lemma}
\newtheorem{cor}[proposition]{Corollary}
\newtheorem{conjec}[proposition]{Conjecture}
\newtheorem{fait}[proposition]{Fact}

\theoremstyle{definition}
\newtheorem{definition}[proposition]{Definition}
\newtheorem{nota}[proposition]{Notation}

\theoremstyle{remark}
\newtheorem{rem}[proposition]{Remark}

\newcommand{\ad}{\operatorname{ad}}

\newcommand{\gk}{\mathfrak{g}}
\newcommand{\hk}{\mathfrak{h}}
\newcommand{\bk}{\mathfrak{b}}
\newcommand{\ak}{\mathfrak{a}}
\newcommand{\ck}{\mathfrak{c}}
\newcommand{\uk}{\mathfrak{u}}
\newcommand{\nk}{\mathfrak{n}}
\newcommand{\mk}{\mathfrak{m}}

\newcommand{\ik}{\mathfrak{i}}

\newcommand{\tk}{\mathfrak{t}}

\newcommand{\K}{\mathbb{K}}

\title{On $p$-Lie algebras of finite Morley rank}
\author{Samuel Zamour}

\begin{document}

\maketitle

\begin{abstract}
We develop the theory of $p$-Lie algebras of finite Morley rank. In particular, we prove a form of weight spaces decomposition relative to an appropriate notion of torus and we obtain a fairly complete characterization in the soluble case.
\end{abstract}

\section{Introduction}
Lie rings have recently attracted renewed interest from model theorists, as evidenced by the articles \cite{DT1}, \cite{DT2} and \cite{Zam2}. An ambitious program for the classification of simple Lie rings of finite Morley rank was proposed in \cite{DT1}; from this perspective, the ultimate goal is to provide a form of blueprint of the classification of simple Lie algebras of finite linear dimension over an algebraically closed field of positive characteristic (see \cite{Stra1}, \cite{Stra2} and \cite{Stra3}). This is therefore the counterpart for Lie rings of the ongoing classification of simple groups of finite Morley rank; the following conjecture, analogous to the Cherlin--Zilber conjecture, is the guiding principle of this research program:
\begin{conjec}\cite[``Logarithmic'' Cherlin--Zilber Conjecture]{DT1}
Let $\gk$ be a simple Lie ring of characteristic $p>0$ of finite Morley rank. If $p$ is sufficiently large, then $\gk$ is isomorphic to a finite-dimensional Lie algebra over a (definable) algebraically closed field of the same positive characteristic.
\end{conjec}
The Morley rank, $MR$, can be thought of as a dimension function on definable sets, generalizing in a certain sense the linear dimension or the Zariski dimension. For groups of finite Morley rank, it admits the following axiomatization:
\begin{itemize}
    \item Axiom A (Monotonicity): Let $A$ be a definable set. Then $MR(A)\geq n+1$ if and only if $A$ contains infinitely many pairwise disjoint definable subsets of rank greater than or equal to $n$.
    \item Axiom B (Definability): Let $A$ and $B$ be definable subsets. If $f: A\rightarrow B$ is a definable function, then for every integer $n$, the set \[ \{ b\in B : MR(f^{-1}(b))=n \}\] is definable.
    \item Axiom C (Additivity): Let $A$ and $B$ be definable subsets. If $f:A\rightarrow B$ is a surjective definable function, and if $MR(f^{-1}(b))=n$ for all $b\in B$, then $MR(A)=n+MR(B)$.
    \item Axiom D (Elimination of the infinite quantifier): 
    Let $A$ and $B$ be definable subsets. For any definable function $f:A\rightarrow B$, there exists an integer $m$ such that for every $b\in B$, the fiber $f^{-1}(b)$ is infinite as soon as it contains at least $m$ elements.
\end{itemize}
For background on groups of finite Morley rank, we refer the reader to \cite{BN} and \cite{ABC} .
\\\\
An essential step in the classification of finite-dimensional modular Lie algebras was the understanding of the structure of simple $p$-Lie algebras (or restricted Lie algebras) (see in particular \cite{BW}). A $p$-Lie algebra is a Lie algebra equipped with a ``$p$-th power map''; for matrix Lie algebras in positive characteristic, this is simply the usual $p$-th power map for matrices. Note that in positive characteristic, the ``classical'' simple Lie algebras associated with algebraic groups are restricted; moreover, a number of simple Lie algebras of ``Cartan type'', which are finite-dimensional analogues in positive characteristic of infinite-dimensional Lie algebras of characteristic zero, are also restricted (for example, the Witt algebra).
\\\\
Similarly, the classification of $p$-Lie algebras of finite Morley rank seems to be a prerequisite to significant progress towards the ``logarithmic'' Cherlin--Zilber conjecture. In this article, we lay the foundations for the study of $p$-Lie algebras of finite Morley rank. In particular, we obtain a fairly complete characterisation in the abelian, nilpotent and soluble cases. We also clarify the relationship between semi-simplicity and unipotency, which seemed to be lacking in the case of abstract Lie rings without $p$-structure. In particular, we prove the following results:
\begin{theo}
Let $\gk$ be a $p$-Lie algebra of finite Morley rank and let $\tk$ be a torus that embeds definably as a $p$-Lie algebra into $(\K,+)$, where $\K$ is a definable field of characteristic $p>0$. Let $V$ be a definable connected $p$-module for $\tk$. Then $V=V^{\tk}\oplus V_1\oplus \dots\oplus V_r$, where each $V_i$ is a definable connected irreducible $\tk$-module (in the definable connected category) and $V^{\tk}=\{v\in V : [\tk,v]=0\}$.
\end{theo}
\begin{theo}
Let $\gk$ be a connected soluble non-nilpotent $p$-Lie algebra of finite Morley rank such that $p>MR(\gk)$. Then the following properties hold:
 \begin{enumerate}
     \item $\gk$ has non-trivial tori.
     \item $\ck$ is a Cartan subalgebra if and only if $\hk=C_{\gk}(\tk)$ where $\tk$ is a non-trivial maximal torus.
     \item $\gk=\gk'+\ck$ for any Cartan subalgebra $\ck$.
\end{enumerate}
\end{theo}
In the second section, we present several algebraic and model-theoretic tools for studying $p$-Lie algebras of finite Morley rank. The third section deals with the structure of abelian and nilpotent $p$-Lie algebras; in particular, we introduce an appropriate notion of ``torus'' and prove its basic properties. In Section 4, we turn to the case of soluble $p$-Lie algebras; we clarify their internal structure and develop the Frattini theory in this context. Finally, in a brief fifth section, we formulate a classification program for ``minimal'' simple $p$-Lie algebras.

\section{Algebraic and model-theoretic tools}
\subsection{Algebraic preliminaries on $p$-Lie algebras}
The fundamental reference for basic results on $p$-Lie algebras is \cite{FS}. We assume some acquaintance with the basic concepts and notations of the theory of Lie algebras (subalgebra, ideal, adjoint map...). 
\\
We begin by giving a precise definition of our objects.
\begin{definition}
    A $p$-Lie algebra (or restricted Lie algebra) is a Lie algebra $\mathfrak{g}$ over a field $\K$ of characteristic $p>0$ equipped with a map $[p]:\mathfrak{g}\rightarrow\mathfrak{g}$, $x\mapsto x^{[p]}$ satisfying the following properties:
    \begin{enumerate}
        \item $[x^{[p]},y]=\ad_{x}^p(y)$ for all $y\in\frak{g}$.
        \item $(\lambda x)^{[p]}=\lambda^p x^{[p]}$ for all $\lambda\in \K$.
        \item $(x+y)^{[p]}=x^{[p]}+y^{[p]}+\sum_{1\leq i \leq p-1}s_i(x,y)$, where $\ad_{x\otimes X+y\otimes 1}^{p-1}(x\otimes 1)=\sum_{1\leq i \leq p-1}is_i(x,y)\otimes X^{i-1}$ in $\mathfrak{g}\otimes_{\K}\K[X]$ for all $x,y\in \mathfrak{g}$.
    \end{enumerate}
\end{definition}
\begin{rem}
Regarding the third axiom, note that there exist explicit formulas for the Jacobson polynomials $s_i(x,y)$, that are expressed in terms of $x,y$ and the language of Lie rings.  
\end{rem}
\begin{definition}
Let $\gk$ be a $p$-Lie algebra. A subalgebra (respectively, an ideal) is called a $p$-subalgebra (respectively, a $p$-ideal) if it is stable under the $p$-map. If $\hk$ is another $p$-Lie algebra, $\phi :\gk\rightarrow\hk$ is a $p$-morphism if it is a morphism of Lie algebras that preserves the $p$-structure.
\end{definition}
\begin{nota}
In the following, we shall write $x^p$ rather than $x^{[p]}$. If $X$ is a subset, we write $X^{p}=X^{[p]}=\{x^p :x\in X\}$ (we use the notation $X^{[p]}$ when that may exist a confusion with the nilpotent series of a subalgebra). Moreover, we may denote the $p$-map by $\phi$.
\\
For convenience, we shall often say $p$-algebra rather than $p$-Lie algebra.
\\
Let $X$ be a subset; $\langle X \rangle$ refers to the subspace generated by $X$.
\end{nota}
\begin{fait}
Let $\gk$ be a $p$-algebra.
\begin{enumerate}
    \item \cite[Lemma 1.2, II]{FS} Let $x,y\in \gk$; then $s_i(x,y)\in \hk^{p}$ (the $p$-th term of the nilpotency series) where $\hk$ is the Lie algebra generated by $x$ and $y$. In particular, if $[x,y]=0$, then $(x+y)^p=x^p+y^p$. Moreover, the restriction of the $p$-map to a nilpotent $p$-subalgebra of nilpotency class less than $p-1$ is additive.
    \item Let $\hk$ be a subalgebra. Then $\hk_p$, the smallest $p$-algebra containing $\hk$, is $\langle \bigcup_{i\in \mathbb{N}} \hk^{[p]^i}\rangle$.
    \item If $\hk$ is soluble (respectively nilpotent), then $\hk_p$ is soluble (respectively, nilpotent) of the same class.
    \item If $\hk$ is an ideal, then $\hk_p$ is an ideal.
\end{enumerate}
\end{fait}
\begin{proof}
    \begin{enumerate}
We give a proof of the last three points.
The subspace $\langle \bigcup_{i\in \mathbb{N}} \hk^{[p]^i}\rangle$ is clearly stable under the $p$-map. It therefore suffices to verify stability under the bracket: \[[x^{p^i},y^{p^j}]=-\ad_x^{p^i-1}\circ \ad_y^{p^j-1}([x,y])\in \hk,\] for all $x,y\in \hk$.
\\\\
We proceed by induction on the class of solubility (respectively, the class of nilpotency) to show that $\hk_p^{(n)}=\hk^{(n)}$ and that $\hk_p^{n}=\hk^n$ (the nilpotent series). For $n=1$, we use once again the equality: \[[x^{p^i},y^{p^j}]=-\ad_x^{p^i-1}\circ \ad_y^{p^j-1}([x,y])\in \hk',\] for all $x,y\in \hk$.
\\\\
Regarding the last property, for all $y\in \gk$ and $x\in \hk$, we have $[x^{p^i},y]=\ad_x^{p^i-1}([x,y])\in \hk$ if $\hk$ is an ideal.
    \end{enumerate}
\end{proof}
\begin{fait}
Let $\gk$ be a $p$-Lie algebra and let $\ik$ be an abelian ideal; then $\ik^p\subseteq Z(\gk)$. In particular, if $\gk$ is soluble and defined over a perfect field, then a minimal ideal, i.e.\ one that contains no non-trivial proper ideal, is an abelian $p$-ideal.
\end{fait}
\begin{proof}
    Let $x\in \ik$ and $y\in \gk$; $[x^p,y]=\ad_x^{p-1}([x,y])=0$ since $\ik$ is an abelian ideal, and therefore $\ik^p\subseteq Z(\gk)$. If $\gk$ is soluble, a minimal ideal $\ik$ is abelian; moreover, the perfectness hypothesis on the base field implies that $\ik^p$ is stable by scalar multiplication and so is an ideal. Finally, either $\ik^p=\ik$, or $\ik^p=0$; in both cases, we deduce that $\ik$ is in fact an abelian $p$-ideal.
\end{proof}
\begin{fait}
Let $\gk$ be a $p$-Lie algebra. The following subalgebras are $p$-subalgebras: $C_{\gk}(X)$, where $X$ is a subset, $N_{\gk}(A)$, where $A$ is a subspace, $Z_n(\gk)$ for $n\in \mathbb{N}$, and the Fitting ideal, $F(\gk)$, when it exists.
\end{fait}
\begin{definition}
A $p$-subalgebra $\uk$ is said to be $p$-nilpotent if for every $x\in\uk$, there exists $i\in \mathbb{N}$ such that $x^{p^i}=0$. It is of bounded exponent if there exists $N$ such that $x^{p^N}=0$ for all $x\in \uk$. More generally, a subspace $V$ is said to be $p$-nilpotent if it is stable under the $p$-map and all its elements are $p$-nilpotent.
\end{definition}
\subsection{Model-theoretic preliminaries}
Recall that a group of finite Morley rank admits a Descending Chain Condition (DCC) on its definable subgroups. In particular, given a group of finite Morley rank $G$, we can define the connected component, $G^{\circ}$ as the intersection of definable subgroups of finite index.
\\
Another basic tool is the following theorem, \textit{the indecomposability theorem}, due to Zilber:
\begin{fait}\cite[Theorem 5.26]{BN}
 Let $(A_i)_i$ be a family of indecomposable subsets of a group $G$ of finite Morley rank. Assume that each $A_i$ contains the identity element $1$ of $G$. Then the subgroup generated by the subsets $A_i$ is definable and connected. Furthermore, there are finitely many indices $i_1,\dots,i_m$ (not necessarily distinct) such that $\langle\bigcup_iA_i\rangle=A_{i_1}\dots A_{i_m}$ for some $m<2^{MR(G)+1}-2$.  
\end{fait}
The article \cite{Zam2} provides a number of cohomological tools for studying Lie rings of finite Morley rank.
\begin{fait}\label{ccohomogie et sous-module}\cite[Theorem 3.7]{Zam2}
    Let $\gk$ be a connected nilpotent Lie ring of finite Morley rank and let $A$ be a definable connected $\gk$-module. Suppose that $A^{\gk}=\{a\in A : g\cdot a=0,~ \text{for all}~g\in \gk\}=0$. Then for all definable connected submodules $C\subseteq B$, we have $(B/C)^{\gk}=0$.
\end{fait}
The structure of soluble Lie rings of finite Morley rank is in part elucidated in \cite{DT2} and \cite{TZ}.
\begin{definition}
A subring $\ck$ is said to be def-abnormal if it is definable and if every definable subring $\uk$ containing $\ck$ is self-normalizing.
\end{definition}

We shall need the following result:
\begin{fait}\label{anneau résoluble}
Let $\gk$ be a connected soluble Lie ring of finite Morley rank such that $\mathrm{char}(\gk)>MR(\gk)$.
\begin{enumerate}
    \item \cite[Proposition 5.4 and Proposition 5.9]{TZ} Let $\hk$ be a definable connected nilpotent subring. Then $E_{\gk}(\hk)=\bigcap_{x\in\hk}E_{\gk}(x)$ is definable, connected and def-abnormal, where $E_{\gk}(x)=\bigcup_{n\in \mathbb{N}}\ad^{-n}_x(0)$. Moreover, there exists $k\in \mathbb{N}$ such that $\ad^k_x(E_{\gk}(\hk))=0$ for all $x\in \hk$.
    \item \cite[Corollary 3.12]{TZ} $\gk$ has Cartan subrings, i.e., definable connected, self-normalizing and nilpotent ones. Moreover, a subring is a Cartan subring if and only if it is def-abnormal minimal if and only if it is Engel minimal.
\end{enumerate}
\end{fait}

In the remainder of the article, a $p$-Lie algebra of finite Morley rank denotes a Lie ring $\langle \gk, +,[,],\phi\rangle$ of finite Morley rank such that $(\gk,+)$ is an elementary abelian $p$-group for a prime number $p>3$(so we have a natural $\mathbb{F}_p$-vector space structure compatible with the bracket, which can be assumed to be definable) and such that $\phi$ denotes a definable $p$-map relative to this Lie algebra structure over $\mathbb{F}_p$. Note that any subgroup is automatically an $\mathbb{F}_p$-vector space.

\section{Abelian and nilpotent $p$-algebras}
\subsection{First properties}
We begin by clarifying the impact of the $p$-structure on definable closures and connected components. A constant difficulty we shall face stems from the fact that the $p$-map is not in general a morphism of additive groups.
\begin{lem}\label{p-ideal connexe}
    Let $\nk$ be a nilpotent $p$-algebra of finite Morley rank of nilpotency class less than $p-1$. Then $\nk^{\circ}$ is a $p$-ideal of $\nk$.
\end{lem}
\begin{proof}
    Note that for any $x\in\nk$ and any $i\in\mathbb{N}$, $\ad_x(\nk^{\circ})$ and $(\nk^{\circ})^{p^i}$ are definable connected groups; it suffices to pass to the connected component.
\end{proof}
\begin{lem}
Let $\gk$ be a $p$-algebra of finite Morley rank. Let $\nk$ be a definable connected subalgebra (respectively, ideal) of nilpotency class less than $p-1$. Then $\nk_p$ is a nilpotent definable connected $p$-algebra (respectively, $p$-ideal).
\end{lem}
\begin{proof}
    The indecomposability theorem yields that $\langle \nk^{p^i} : i\in \mathbb{N}\rangle=\nk_p$ is definable and connected.
\end{proof}
\begin{lem}
    Let $\gk$ be a $p$-algebra of finite Morley rank and let $\ak$ be an abelian $p$-subalgebra. Then $d(\ak)$, the intersection of all definable subgroups containing $\ak$, is a definable abelian $p$-subalgebra.
\end{lem}
\begin{proof}
    Note that $Z(C_{\gk}(\ak))$ is a definable abelian $p$-subalgebra containing $\ak$ and therefore also $d(\ak)$. Moreover, $\phi^{-1}(d(\ak))\cap Z(C_{\gk}(\ak))$ contains $\ak$ and therefore also $d(\ak)$.
\end{proof}
We now seek to prove a decomposition theorem for abelian $p$-algebras as a direct sum of a ``toral'' part and a ``unipotent'' part.
\begin{definition}
An abelian $p$-algebra $\tk$ is a $p$-torus if $\tk$ is $p$-divisible, i.e., $\tk^p=\tk$. If $\tk$ is moreover connected of finite Morley rank, we say that $\tk$ is a torus.
\end{definition}
\begin{lem}
A connected $p$-algebra $\tk$ of finite Morley rank is a torus if and only if $\phi|_{\tk}$ is surjective if and only if $\phi|_{\tk}$ has a finite kernel.
\end{lem}
\begin{proof}
Since $\phi$ defines a group morphism on $\tk$, we have $MR(\tk)=MR(\ker(\phi))+MR(\mathrm{im}(\phi))$ and the conclusion follows.
\end{proof}
Analogously to groups, we introduce a notion of purity relative to the $p$-map.
\begin{definition}
   Let $\ak$ be an abelian $p$-Lie algebra. A $p$-subalgebra $\bk$ is said to be pure if $\ak^{p^n}\cap\bk=\bk^{p^n}$ for all $n\in \mathbb{N}$.
\end{definition}
\begin{lem}\label{pureté décomposition}
Let $\ak$ be an abelian $p$-algebra (over $\mathbb{F}_p$), let $\bk$ be a pure $p$-subalgebra, and let $x$ be such that $x^{p^n}\in \bk$ but $x^{p^{n-1}}\notin \bk$. Then $\langle \bk,x\rangle_p=\bk\oplus \ck$, where $\ck=\langle y \rangle_p=\langle y,y^p,\dots, y^{p^{n-1}} \rangle$ with $y=(x-b_0)$ for $b_0^{p^n}=x^{p^n}$.
\end{lem}
\begin{proof}
 By purity of $\bk$, there exists $b_0\in \bk$ such that $b_0^{p^n}=x^{p^n}$; set $y=(x-b_0)$. Note that $y^{p^n}=0$ by construction; thus, $\langle y\rangle_p=\langle y, y^p,\dots, y^{p^{n-1}}\rangle=\ck$. We first show that $\ck\cap \bk=0$. Let $a_0\cdot y+\dots+a_{i}\cdot y^{p^{i}}\in \ck\cap \bk$, $i\leq n-1$, $a_0,\dots,a_{i}\in \mathbb{F}_p$ and $a_{i}\neq 0$; then $a_0\cdot x+\dots+a_{i}\cdot x^{p^{i}}\in \bk$. We may assume that this sum contains another non-zero term of the form $a_j\cdot x^{p^j}$, with $j<i$. Applying the $p$-map $(n-i)$ times, since $x^{p^n}\in \bk$, we reduce to considering a sum of the form $a_0\cdot x^{p^{n-i}}+\dots+a_j\cdot x^{p^{j+n-i}}$ of strictly smaller length. Finally, by iterating this process, we show that there exists $i<n$ such that $x^{p^i}\in\bk$, a contradiction.
 \\
 Moreover, since $x=y+b_0$, it is clear that $\bk+\langle x\rangle_p=\bk+\langle y \rangle_p$.
\end{proof}
\begin{proposition}\label{pureté} (compare with \cite[Lemma 1.15, I]{ABC})
 Let $\ak$ be an abelian $p$-Lie algebra (over $\mathbb{F}_p$) and let $\bk$ be a pure $p$-subalgebra such that $\ak/\bk$ is $p$-nilpotent of bounded exponent. Then $\ak=\bk\oplus \ck$, where $\ck$ is a $p$-nilpotent $p$-subalgebra of bounded exponent.
\end{proposition}
\begin{proof}
By Zorn's lemma, we can find a pure $p$-subalgebra $\mk$ containing $\bk$ and such that $\mk/\bk$ is split over $\bk$, i.e., there exists a $p$-subalgebra $\mathfrak{d}$ such that $\mk=\bk\oplus \mathfrak{d}$, maximal with respect to these properties. Suppose for contradiction that $\ak\neq \mk$. Let $x\in \ak$ such that $x^{p^n}\in \mk$ and $x^{p^{n-1}}\notin\mk$ with $n$ maximal. Since $\mk$ is pure, by Lemma \ref{pureté décomposition}, we can write $\langle \mk,x\rangle_p=\mk\oplus\ck$, where $\ck=\langle y,\dots, y^{p^{n-1}}\rangle$ with $y=(x-b_0)$ for $b_0^{p^n}=x^{p^n}$. We may also assume that $y,\dots y^{p^{n-1}}$ are linearly independent over $\mathbb{F}_p$.
\\
We show that $\mk\oplus \ck$ is pure; by induction, it suffices to prove that $\ak^p\cap(\mk\oplus \ck)=\mk^p\oplus \ck^p$. Let $z\in \ak$ and suppose $z^p=m+c=m+a_0\cdot y+\dots+a_{n-1}\cdot y^{p^{n-1}}$ with $m\in\mk$ and $c\in \ck$; by virtue of the maximality of $n$, $z^{p^i}\in \mk$ for some $i\in\{1,\dots,n\}$. Consequently, $z^{p^i}=m'+a_0\cdot y^{p^{i-1}}+\dots+a_{n-i}\cdot y^{p^{n-i}}\in \mk$ where $m'\in \mk$; thus, $a_0=\dots=a_{n-i}=0$. Finally, we deduce that $z^p=m+a_{n-i+1}\cdot y^{p^{n-i+1}}+\dots+a_{n-1}\cdot y^{p^{n-1}}$ and we are done by using the purity of $\mk$, a contradiction.
\end{proof}
\begin{rem}
The preceding results are in fact valid over an arbitrary field of characteristic $p>0$.
\end{rem}
\begin{cor}\label{dec abélien}
Let $\ak$ be an abelian $p$-Lie algebra of finite Morley rank. Then the following statements hold:
\begin{enumerate}
    \item $\ak=\uk\oplus \tk$ where $\uk$ is a $p$-nilpotent $p$-subalgebra of bounded exponent, and $\tk$ is a $p$-torus.
    \item $\ak=\uk(+)\tk$ where $\tk$ is a $p$-torus, $\uk$ is a $p$-nilpotent $p$-algebra of bounded exponent, $\uk$ and $\tk$ being definable.
\end{enumerate}
\end{cor}
\begin{proof}
    By the DCC on definable subgroups, there exists $N\in \mathbb{N}$ such that the sequence $\mathrm{im}(\phi^n)$ stabilises from rank $N$ onwards. We set $\tk=\mathrm{im}(\phi^N)$; since $\ak/\tk$ is of exponent $N$ and since $\tk$ is pure by construction, we are done using the previous proposition.
    \\
    For the second assertion, we can give a more direct proof that avoids Proposition \ref{pureté}. Since $\mathrm{im}(\phi^N)=\mathrm{im}(\phi^{2N})$, for any $x$ there exists $y$ such that $\phi^N(x)=\phi^{2N}(y)$; thus, $(x-\phi^N(y))\in \ker(\phi^N)$. Consequently, $x=(x-\phi^N(y))+\phi^N(y)$ and therefore $\ak=\ker(\phi^N)+\mathrm{im}(\phi^N)$; the intersection is finite because $MR(\ak)=MR(\ker(\phi^N))+MR(\mathrm{im}(\phi^N))=MR(\ker(\phi^N))+MR(\mathrm{im}(\phi^N))-MR(\ker(\phi^N)\cap \mathrm{im}(\phi^N))$.
\end{proof}
We deduce two corollaries that generalize the known situation for groups of finite Morley rank.
\begin{cor}
    Let $\gk$ be a $p$-algebra of finite Morley rank that contains no non-trivial $p$-nilpotent elements. Then for every $x\in \gk$ there exists a unique element $y$ such that $y^p=x$.
\end{cor}
\begin{proof}
    Let $x\in \gk$; consider the definable abelian $p$-algebra $\ak=d(x)$. Since $\ak$ contains no non-trivial $p$-nilpotent elements, $\ak$ is $p$-divisible by Corollary \ref{dec abélien}. Thus, $x=y^p$ for some $y\in \ak\subseteq Z(C_{\gk}(x))$. If $z$ satisfies $z^p=x$, then $z\in C_{\gk}(x)$ because $[z^p,z]=\ad_z^{p-1}([z,z])=0$; thus, $[z,y]=0$. In particular, $(y-z)^p=y^p-z^p=0$ and therefore $y=z$.
\end{proof}
\begin{cor}
    Let $\gk$ be a $p$-algebra of finite Morley rank. Let $f:\hk\rightarrow \mk$ be a definable $p$-morphism. If $x\in \hk$ and $f(x)$ is $p$-nilpotent, then there exists a $p$-nilpotent element $x'\in \hk$ such that $f(x)=f(x')$.
\end{cor}
\begin{proof}
Consider $\ker(f)$; there exists an integer $n$ such that $x^{p^n}\in \ker(f)$. Now, $d(x^{p^n})=\ck\oplus \tk\subseteq \ker(f)$ with $\tk$ a $p$-torus and $\ck$ $p$-nilpotent. Thus, $x^{p^n}=t+c$; since $\tk$ is $p$-divisible, there exists $t'\in\tk$ such that $(t')^{p^n}=t$ and consequently, $(x-t')^{p^n}=c$. It suffices to take $x'=x-t'$ to conclude.
\end{proof}

\subsection{Tori and semi-simplicity}
In this subsection, we prove decomposition results for definable $p$-modules relative to the action of a $p$-divisible and nilpotent $p$-algebra of finite Morley rank.
\begin{definition}
    Let $\gk$ be a $p$-Lie algebra of finite Morley rank and let $V$ be a definable connected elementary abelian $p$-group. We say that $V$ is a $p$-module for $\gk$ if there exists a $p$-morphism from $\gk$ to $\mathrm{End}_{\mathrm{def}}(V)$.
\end{definition}
We have an analogue of Maschke's theorem in our context.
\begin{fait}\label{Mas}\cite[Corollary 3.13]{Zam2}
Let $\tk$ be a torus and let $V$ be a definable connected $p$-module for $\tk$. Suppose that $V^{\tk}=0$. Then $V=V_1\oplus\dots\oplus V_r$, where each $V_i$ is a definable connected irreducible $\tk$-module (in the definable connected category).
\end{fait}
The following proposition will play an important role in proving our structure theorems.
\begin{proposition}\label{cohomologie}
    Let $\nk$ be a connected $p$-divisible and nilpotent $p$-algebra of finite Morley rank and let $V$ be a definable connected $p$-module for $\nk$. Then $V=V^{\nk}+[\nk,V]$.
\end{proposition}
\begin{proof}
First, by the indecomposability theorem, the submodule $[\nk, V]$ is definable and connected. We prove that $(V/V^{\nk})^{\nk}=0$. Let $v\in V$ be such that $[n_1,[n_2,v]]=0$ for all $n_1,n_2\in \nk$; then, in particular, for all $n\in \nk$, $\ad^2_n(v)=0$ and therefore $\ad_n^p(v)=[n^p,v]=0$ and the conclusion follows from $p$-divisibility. Suppose for contradiction that $V/(V^{\nk}+[\nk,V])\neq 0$. Now, $(V/V^{\nk})/(V^{\nk}+[\nk,V]/V^{\nk})\simeq V/(V^{\nk}+[\nk,V])$, but $\nk$ acts trivially on $V/(V^{\nk}+[\nk,V])$, a contradiction by Fact \ref{ccohomogie et sous-module}.
\end{proof}
\begin{cor}\label{rigidity torus}
Let $\gk$ be a $p$-algebra of finite Morley rank and let $\tk$ be a torus. Then $N_{\gk}(\tk)^{\circ}=C_{\gk}(\tk)^{\circ}$.
\end{cor}
\begin{proof}
 It suffices to consider the $\tk$-module $N_{\gk}(\tk)^{\circ}$ and to apply Proposition \ref{cohomologie}.
\end{proof}

\begin{cor}\label{espace de poids}
Let $\tk$ be a torus and let $V$ be a definable connected $p$-module for $\tk$. Suppose that $V^{\tk}\cap[\tk,V]=0$. Then $V=V^{\tk}\oplus V_1\oplus\dots\oplus V_r$. Moreover, there exist definable fields $\K_i$ and definable $p$-morphisms $\rho_i : \tk\rightarrow \K_i$, where $i\in\{1,\dots,r\}$.
\end{cor}
The preceding result can be viewed as a weight spaces decomposition. We shall strengthen our notion of torus and we shall see that in this case we obtain a weight spaces decomposition for $p$-modules in an unconditional manner.
\begin{definition}
    \begin{enumerate} Let $\gk$ be a $p$-algebra of finite Morley rank.
        \item An element $x$ is semi-simple if there exists $n$ such that $x^{p^n}=x$.
        \item A torus $\tk$ is said to be $p$-good if every definable $p$-subalgebra is the definable closure (in the usual group-theoretic sense) of its semi-simple elements. We denote by $\tk_s$ the set of semi-simple elements of $\tk$.
    \end{enumerate}
\end{definition}
Note the parallelism with the notion of "good torus" in the context of groups of finite Morley rank. We now give a natural example of a $p$-good torus.
\\
Let $\K$ be an algebraically closed field of characteristic $p>0$. The structure $(\K,+,\phi)$, where $\phi$ is the Frobenius automorphism, is an abelian $p$-Lie algebra over $\mathbb{F}_p$. In this case, the semi-simple  elements correspond exactly to the fixed points under the iterates of the Frobenius automorphism, that is, to the elements of $\overline{\mathbb{F}_p}=\bigcup_{n\in\mathbb{N}}\mathbb{F}_{p^n}$. In this context, we obtain a particularly convenient algebraic characterisation of the $p$-subalgebras of $(\mathbb{F}_{p^n},+,\phi)$. In fact, the latter is a $\mathbb{F}_p[X]$-module for the action: $P(X)\cdot x=P(\phi(x))$. Thus, a $p$-subalgebra is precisely a $\mathbb{F}_p[X]$-submodule relative to this action.
\begin{lem}\label{module cyclique}
In a field of characteristic $p>0$, finitely many semi-simple elements generate a cyclic $\mathbb{F}_p$-module.
\end{lem}
\begin{proof}
Let $t_1\dots,t_n$ be semi-simple elements. We may assume that $t_1,\dots,t_n$ belong to $\mathbb{F}_{p^k}$. Consider the $\mathbb{F}_p[X]$-module they generate. Since the extension $\mathbb{F}_{p^k}/\mathbb{F}_p$ is Galois, by the normal basis theorem, there exists a cyclic basis $\{\alpha,\phi(\alpha)\dots,\phi^{k-1}(\alpha)\}$ for the $\mathbb{F}_p$-vector space $\mathbb{F}_{p^k}$. Thus, $\mathbb{F}_{p^k}$ is a cyclic $\mathbb{F}_p[X]$-module. Therefore every submodule is also cyclic (since $\mathbb{F}_p[X]$ is a principal ideal ring); in other words, there exists $t_0$ such that $\langle t_0 \rangle_p=\langle t_1,\dots,t_n\rangle_p$.
\end{proof}
\begin{proposition}
Let $\K$ be a field of finite Morley rank of positive characteristic $p>0$. Then $\langle \K,+,\phi\rangle$, where $\phi$ is the Frobenius automorphism, is a $p$-good torus.
\end{proposition}
\begin{proof}
Let $\tk$ be a definable $p$-subalgebra and let $\tk_0=d(\tk_{s})$; note that $\tk_0$ is a $p$-algebra since $\tk_{s}$ is a $p$-algebra. Now, the Frobenius automorphism $\phi$ is an automorphism of the structure $\langle \K, \tk,\tk_0\rangle$. Since $\overline{\mathbb{F}_p}$ is a $p$-good torus (every element is semi-simple), the conclusion follows from Wagner's theorem \cite[Proposition 7]{WagF}.
\end{proof}

\begin{cor}(compare with \cite[Lemma 1.14]{CD})
Let $\tk$ be a torus that embeds definably as a $p$-algebra into $(\K,+)$, where $\K$ is a definable field of characteristic $p>0$. Let $V$ be a definable connected $p$-module for $\tk$. Then $V=V^{\tk}\oplus V_1\oplus \dots\oplus V_r$, where each $V_i$ is a definable connected irreducible $\tk$-module.
\end{cor}
\begin{proof}
By Corollary \ref{espace de poids}, it suffices to show that $[\tk,V]\cap V^{\tk}=0$. Now, $V^{\tk}=V^{d(\tk_{s})}=V^{\langle t_1,\dots t_k\rangle_p}$ for finitely many semi-simple elements $t_1, \dots t_k\in \tk$ by virtue of the DCC on definable subgroups. By Lemma \ref{module cyclique}, there exists a semi-simple generator $t_0$ for $\langle t_1,\dots t_k\rangle_p$ and therefore $V^{\tk}=V^{t_0}$. 
\\
Similarly to \cite[Lemma 1.14]{CD}, we also obtain that $[\tk,V]=[t_0,V]$. Note that $[\tk_s,V]$ is a definable connected group and so $\Sigma=\{t\in \tk:[t,V]\subseteq[\tk_s,V]\}$ is a definable subgroup containing $\tk_s$. Since $d(\tk_s)=\tk$, $\Sigma=\tk$ and thus $[\tk,V]=[\tk_s,V]$. Note that $\bigcup_{i\in\mathbb{N}}(\tk\cap\mathbb{F}_{p^i})=\bigcup_{i\in \mathbb{N}}\tk_i=\tk_s$. By the ascending chain condition on definable connected subgroups, the family $\{[\tk_{i},V]:{i\in\mathbb{N}}\}$ admits a maximal element $[\tk_n,V]$. But for any $m\in \mathbb{N}$, $[\tk_m,V],[\tk_n,V]\subseteq [\tk_{lcm(m,n)},V]$ and so $[\tk_m,V]\subseteq [\tk_n,V]$ by maximality. Thus $[\tk_n,V]=[\tk_s,V]=[\tk,V]$. Now, take a semi-simple generator $t_1$ for the $p$-algebra $\tk_n$; then $[t_1,V]=[\tk,V]$ since $V$ is a $p$-module for $\tk$. By taking $n$ sufficiently large, we have $t_0\in \langle t_1 \rangle_p$ and so we may suppose that $t_0=t_1$.
\\
We therefore reduce to studying $[t_0,V]\cap V^{t_0}$. We may assume that $t_0^{p^n}=t_0$. But $\ad_{t_0}^2(x)=0$ implies $[t_0^{p^{n}},x]=[t_0,x]=0$ for all $x\in V$ and we are done.
\end{proof}

\subsection{The nilpotent case}
In general, one cannot decompose a nilpotent $p$-algebra of finite Morley rank as a central product of a torus and a $p$-nilpotent $p$-algebra. Nevertheless, we have the following result:
\begin{proposition}
    Let $\nk$ be a connected nilpotent $p$-Lie algebra of finite Morley rank. Then every torus is central.
\end{proposition}
\begin{proof}
We proceed by induction on the rank of $\nk$. Let $\tk$ be a torus of $\nk$. We may assume that $\nk$ is not abelian and we apply the induction hypothesis to $\nk/Z(\nk)$. Thus, $\overline{\tk}\subseteq Z(\nk/Z(\nk))$; in other words, $\tk\subseteq Z_2(\nk)$. But by Proposition \ref{cohomologie}, $\nk=\nk^{\tk}+[\nk,\tk]=C_{\nk}(\tk)+[\nk,\tk]$; now $[\nk,\tk]\subseteq Z(\nk)$ and therefore $\tk\subseteq Z(\nk)$.
\end{proof}
In the case where the $p$-algebra is ``sufficiently'' nilpotent, we obtain stronger structure theorems.
\begin{lem}\label{tore du fitting}
  Let $\gk$ be a connected $p$-algebra of finite Morley rank. If a torus $\tk$ is contained in a nilpotent ideal of nilpotency class less than $p-1$, then $\tk$ is central in $\gk$.
\end{lem}
\begin{proof}
    Since $\gk=\gk^{\tk}+[\tk,\gk]$, it suffices to verify that $[\tk,\gk]=0$. Now, $\ad_t^p(x)=\ad_t^{p-1}([t,x])=[t^p,x]=0$ for all $t\in\tk$ and $x\in \gk$; the conclusion follows from $p$-divisibility of $\tk$.
\end{proof}
\begin{proposition}\label{structure nilpotente}
    Let $\nk$ be a connected nilpotent $p$-algebra of nilpotency class less than $p-1$ of finite Morley rank. Then $\nk=\tk(+)\uk$, where $\tk$ is a central torus and $\uk$ is a $p$-nilpotent $p$-subgroup.
\end{proposition}
\begin{proof}
    Since $\nk$ is nilpotent of nilpotency class less than $p-1$, $\nk^p=\phi(\nk)$ is a $p$-subalgebra contained in $Z(\nk)$. If $\nk^p=0$, then $\nk$ is $p$-nilpotent; otherwise, since $\phi$ is additive, reasoning as in the proof of Corollary \ref{dec abélien}, we obtain a positive integer $i\geq 1$ such that $\nk=im(\phi^i)(+)\ker(\phi^i)$. Note that $im(\phi^i)$ is a central torus and $\ker(\phi^i)$ is a $p$-nilpotent $p$-group.
\end{proof}

\section{Soluble $p$-algebras}

\subsection{Structure theorems}
We study the distribution of tori in the case of a soluble $p$-algebra of finite Morley rank. We take inspiration from \cite{Win} and we generalize some results of this paper to our context.
\begin{proposition}\label{critère nilpotence}
    Let $\gk$ be a connected soluble $p$-algebra of finite Morley rank such that $p>MR(\gk)$. Suppose there is no non-trivial torus. Then $\gk$ is nilpotent.
\end{proposition}
\begin{proof}
First, note that every connected nilpotent subring has nilpotency class less than $p-1$. Let $\ck$ be a Cartan subalgebra (which exists by Fact \ref{anneau résoluble}; since a Cartan subring is self-normalizing, it is in fact a $p$-algebra). The hypothesis on tori implies that $Z(\ck)^{\circ}$ is $p$-nilpotent. Now, $\ck^p\subseteq Z(\ck)^{\circ}$ : $[x^p,y]=\ad_x^p(y)=0$ for all $x,y\in \ck$, since $\ck$ is nilpotent of nilpotency class less than $p-1$. Thus, $\ck$ is $p$-nilpotent. By \cite[Lemma 5.3]{TZ}, $F(\gk)$ contains all ad-nilpotent elements and in particular all $p$-nilpotent elements. Consequently, $\ck\subseteq F(\gk)^{\circ}$ (since $\ck$ is connected); since $\ck$ is self-normalizing, the normalizer condition yields $F^{\circ}(\gk)=\ck=\gk$.
\end{proof}

\begin{lem}\label{maximalité}
    Let $\gk$ be a connected $p$-algebra of finite Morley rank and let $\ik$ be a definable connected $p$-ideal. Suppose that $\gk/\ik$ and $\ik$ are tori. Then $\gk$ is a torus.
\end{lem}
\begin{proof}
   Note that $\ik$ is central by Lemma \ref{tore du fitting}. Thus, $\gk$ is 2-nilpotent. In particular, $\gk^p$ is a central $p$-ideal. Note that $\gk$ is $p$-divisible; indeed, for every $y\in \gk$, there exists $x\in \gk$ such that $y-x^p\in \ik$, since this ideal is $p$-divisible, and since $\gk$ is 2-nilpotent, we deduce that $\gk^p=\gk$.
\end{proof}
\begin{proposition}\label{centralisateur nilpotent}
    Let $\gk$ be a connected $p$-algebra of finite Morley rank such that $p\geq MR(\gk)$. Suppose that $C_{\gk}^{\circ}(\tk)$ is a proper soluble subalgebra, where $\tk$ is a maximal torus. Then $C_{\gk}^{\circ}(\tk)$ is a nilpotent $p$-algebra and is almost self-normalizing.
\end{proposition}
\begin{proof}
    By Corollary \ref{rigidity torus}, it suffices to verify nilpotency. Consider $\hk=C_{\gk}(\tk)^{\circ}/\tk$. Now, $\hk$ is soluble and $p>MR(\hk)$. If $\hk$ contains a non-trivial torus $\overline{\tk_1}$, then $\tk_1$ is a torus of $C_{\gk}(\tk)^{\circ}$ strictly containing $\tk$ by Lemma \ref{maximalité}, a contradiction. By Lemma \ref{critère nilpotence}, $\hk$ is nilpotent and the same holds for $C_{\gk}^{\circ}(\tk)$.
\end{proof}
We shall characterize the structure of soluble $p$-algebras of finite Morley rank in terms of the centralizers of maximal tori; we begin with two lemmas.
\begin{lem}\label{centralisateur généralisé torus}
Let $\gk$ be a connected soluble $p$-algebra of finite Morley rank such that $p>MR(\gk)$, and let $\tk$ be a torus. Then $E_{\gk}(\tk)=C_{\gk}(\tk)$. In particular, $C_{\gk}(\tk)$ is connected.
\end{lem}
\begin{proof}
Since, by Fact \ref{anneau résoluble}, $\tk\subseteq F^{\circ}(E_{\gk}(\tk))$, any element of $\tk$ acts as a nilpotent operator of nilpotency class less than $p-1$ on $E_{\gk}(\tk)$. In particular, $[t^p,x]=0$ for all $t\in\tk$ and all $x\in E_{\gk}(\tk)$; the conclusion follows from $p$-divisibility.
\end{proof}
\begin{lem}\label{centralisateur généralisé}
 Let $\gk$ be a connected soluble $p$-algebra of finite Morley rank such that $p>MR(\gk)$. Let $\hk$ be a definable connected nilpotent $p$-subalgebra containing a non-trivial torus $\tk$ that is maximal (in $\hk$). Then $E_{\gk}(\hk)=E_{\gk}(\tk)$.
\end{lem}
 \begin{proof}
By Proposition \ref{structure nilpotente}, there exists an integer $n$ such that $\hk^{p^n}=\tk$. Thus, let $x\in \hk$; then $x^{p^n}\in \tk$. Consequently, $x^{p^n}$ acts as a nilpotent operator (of nilpotency class $k$, say) on $E_{\gk}(\tk)$ and the same holds for $x$: $\ad^k_{x^{p^n}}(y)=\ad_x^{kp^n}(y)=0$ for all $y\in E_{\gk}(\tk)$.
 \end{proof}

\begin{theo}
 Let $\gk$ be a connected soluble non-nilpotent $p$-algebra of finite Morley rank, such that $p>MR(\gk)$. Then the following properties hold:
 \begin{enumerate}
     \item $\ck$ is a Cartan subalgebra if and only if $\hk=C_{\gk}(\tk)$ where $\tk$ is a non-trivial maximal torus.
     \item $\gk=\gk'+\ck$ for any Cartan subalgebra $\ck$.
 \end{enumerate}
\end{theo}
\begin{proof}
    Let $\tk\neq 0$ be a non-trivial maximal torus (which exists by of Proposition \ref{critère nilpotence}). We already know that $C_{\gk}(\tk)=E_{\gk}(\tk)$ is nilpotent (Proposition \ref{centralisateur nilpotent}) and it is self-normalizing by Fact \ref{anneau résoluble}.
    \\
    Conversely, let $\ck$ be a Cartan subalgebra. Arguing as in the proof of Proposition \ref{critère nilpotence}, observe that it contains a non-trivial torus $\tk$ which we may assume to be maximal in $\ck$. But Lemma \ref{centralisateur généralisé} and Lemma \ref{centralisateur généralisé torus} combined with Fact \ref{anneau résoluble} yields that $C_{\gk}(\tk)=E_{\gk}(\tk)=E_{\gk}(\ck)=\ck$; since $\ck$ is self-normalizing, the torus $\tk$ is necessarily maximal in $\gk$.
    \\\\
    For the second point, it suffices to apply Proposition \ref{cohomologie}.
\end{proof}
At this stage, the question of the conjugacy of tori is quite natural; however, even in the favourable case of a 2-ad-nilpotent element $x$, there is no reason why $\exp(x)=\mathrm{Id}+\ad_x$, the Lie algebra automorphism it induces, should preserve the $p$-structure. In the case of a finite-dimensional $p$-algebra, Winter introduces a new exponentiation operator and proves the conjugacy of tori in the soluble case relative to this operator. Nevertheless, this strategy relies on rather strong rigidity properties of the weight spaces decomposition, which are lacking in our context (Theorem \ref{espace de poids} is too rudimentary for this purpose). We do, however, have some control over the structure of tori and their dimension (we adapt the results obtained by Winter in \cite{Win}).
\begin{proposition}\label{tore quotient}(Compare with \cite[Theorem 2.16]{Win})
    Let $\gk$ be a connected soluble $p$-algebra of finite Morley rank such that $p> MR(\gk)$. Let $\tk$ be a maximal torus and $\ik$ a definable connected $p$-ideal. Then $\tk+\ik/\ik$ is a maximal torus of $\gk/\ik$ and all maximal tori are of this form.
\end{proposition}
\begin{proof}
    Suppose that $\overline{\ak}$ is a maximal torus of $\overline{\gk}=\gk/\ik$ and let $\tk_1$ be a maximal torus of $\bk=\pi^{-1}(\overline{\ak})$. Then $\bk=\bk'+C_{\bk}(\tk_1)=\bk'+\ck$, where $\ck=C_{\bk}(\tk_1)$ is a Cartan subalgebra. But $\pi(\bk)$ is abelian, so $\pi(\bk')=0$ and thus $\pi(\ck)=\pi(\bk)=\overline{\ak}$. But $\ck/\tk_1$ is $p$-nilpotent by Proposition \ref{structure nilpotente}; thus, $\pi(\ck)/\pi(\tk_1)=\overline{\ak}/\pi(\tk_1)$ is $p$-nilpotent and $p$-divisible, and therefore $\pi(\tk_1)=\overline{\ak}$. Consequently, $\tk_1$ is a maximal torus of $\gk$. Indeed, let $\tk_2$ be a maximal torus containing $\tk_1$; then $\pi(\tk_2)$ is a torus containing $\overline{\ak}$, so $\pi(\tk_2)=\overline{\ak}$ and thus $\tk_2\subseteq\pi^{-1}(\overline{\ak})=\bk$, hence $\tk_2=\tk_1$.
\end{proof}
\begin{cor}(Compare with \cite[Proposition 2.17]{Win})
    Let $\gk$ be a connected soluble $p$-algebra of finite Morley rank such that $p>MR(\gk)$. Then the rank of a maximal torus is constant.
\end{cor}
\begin{proof}
    We proceed by induction on $MR(\gk)$. The result is trivial if $\gk$ is abelian or if $MR(\gk)=1$. Let $\ik$ be a $\gk$-minimal definable connected $p$-ideal; it is either a $p$-divisible or a $p$-nilpotent abelian ideal. We consider $\gk/\ik$ and let $\tk_1$ and $\tk_2$ be two maximal tori. By induction, $MR(\pi(\tk_1))=MR(\pi(\tk_2))$, since $\pi(\tk_k)$ is a maximal torus of $\gk/\ik$ by Proposition \ref{tore quotient}, for $k=1,2$. If $\ik$ is $p$-nilpotent, then $\tk_k\cap\ik$ is finite, for $k=1,2$; if $\ik$ is $p$-divisible, it is central and $\ik\subseteq \tk_k$, for $k=1,2$ (just consider the nilpotent $p$-algebras $C_{\gk}(\tk_k)$).
\end{proof}

\subsection{Frattini theory}
In this subsection, we develop the Frattini theory in the context of soluble $p$-algebras of finite Morley rank. For the finite-dimensional (linear) case, we refer the reader to \cite{LT1} and \cite{LT2}; we draw inspiration from the proof strategies used in these two papers.
\begin{definition}
We define the $p$-Frattini subalgebra, $\Phi_p(\gk)$, as the intersection of the maximal definable connected $p$-subalgebras.
\end{definition}
First, we shall prove an analogue of the following fact:
\begin{fait}\cite[Proposition 4.3]{TZ}
Let $\gk$ be a connected soluble Lie ring of finite Morley rank such that $\gk'$ is nilpotent. Then $\Phi(\gk)$, the intersection of the maximal definable connected subrings, is an ideal.
\end{fait}
\begin{lem}\label{maximalité p-algèbre}
Let $\gk$ be a connected $p$-algebra of finite Morley rank and let $\mk$ be a maximal definable connected subalgebra. Suppose that $\mk$ is not an ideal of $\gk$ and $\mk$ is abelian. Then $\mk$ is a $p$-subalgebra.
\end{lem}
\begin{proof}
Suppose for contradiction that $\mk$ is not a $p$-algebra. Note that $\mk_p$ is a definable connected subalgebra, so by maximality, we have $\gk= \mk_p$. Consequently, $\gk'=\mk_p'\subseteq \mk'=0$, a contradiction.
\end{proof}
\begin{proposition}(Comparer with \cite[Theorem 2.3]{LT2})
  Let $\gk$ be a connected soluble $p$-algebra of finite Morley rank such that $p>MR(\gk)$. Then $\Phi_p(\gk)$ is a $p$-ideal.
\end{proposition}
\begin{proof}
  The proof is very close to that of \cite[Proposition 4.3]{TZ}. First, note that $\gk$ has a $\gk$-minimal definable connected abelian $p$-ideal $\ik$ contained in $\gk'$ and that $C_{\gk}(X)$ is a $p$-algebra, for every subset $X$. We then reduce to the following situation: $\ik$ is of finite index in its centralizer and $\gk=\ik\oplus \mk$, where $\mk$ is a maximal definable connected $p$-subalgebra, which does not contain $\ik$ and which is abelian. Since $\ik$ is $\mk$-minimal, $\mk$ is in fact a maximal definable subalgebra: let $\mk_1$ be a definable subalgebra containing $\mk$; then $\mk_1=(\ik\cap \mk_1)\oplus \mk$ and therefore either $\mk_1=\mk$ or $\mk_1=\gk$.
  \\\\
  We claim that $\Phi_p(\gk)$ is contained in the finite $p$-ideal $C_{\mk}(\ik)$. Indeed, let $x\in \mk/C_{\mk}(\ik)$; there exists $a\in \ik$ such that $[a,x]\neq 0$. In particular, $\ad_a$ is 2-nilpotent and $\exp(a)=\mathrm{Id}+\ad_a$ is an automorphism of the Lie algebra structure (but does not a priori preserve the $p$-structure). Observe that $\exp(a)(\mk)$ is a maximal definable connected subalgebra. Suppose for contradiction that $\ik\subseteq \exp(a)(\mk)$; then for any $b\in \ik$, there exists $m\in \mk$ such that $b=m+[a,m]$, so $m\in \mk\cap \ik=0$, a contradiction. Consequently, $\exp(a)(\mk)$ does not contain $\gk'$ and thus $\exp(a)(\mk)$ is not an ideal (otherwise, $\gk/\exp(a)(\mk)$ would be abelian by maximality). By Lemma \ref{maximalité p-algèbre}, $\exp(a)(\mk)=\mk_1$ is in fact a $p$-subalgebra. If $x\in\mk_1\cap\mk$, then $x=y+[a,y]$ with $y\in \mk$, thus $[a,y]\in \ik\cap\mk$. But $\ik\cap \mk=0$, and therefore $[a,y]=0$. This implies that $x=y$ and $[a,x]=0$, a contradiction. Consequently, $x\notin \Phi_p(\gk)$ and the claim is proved. Since $C_{\mk}(\ik)$ is a finite ideal, it is central and thus $\Phi_p(\gk)\lhd \gk$.
\end{proof}
\begin{cor}
Let $\gk$ be a connected soluble $p$-algebra of finite Morley rank such that $p>MR(\gk)$. Then $\Phi_p(\gk)^{\circ}$ is a nilpotent $p$-ideal.
\end{cor}
\begin{proof}
Note that Cartan subrings are in fact $p$-subalgebras and argue as in the proof of \cite[Proposition 4.5]{TZ}. For the stability under the $p$-map, it suffices to use Lemma \ref{p-ideal connexe}.   
\end{proof}
The following two lemmas are easy adaptations of results established by Towers in \cite{T}.
\begin{lem}\label{Frattini 1}
Let $\gk$ be a connected Lie ring (respectively, $p$-algebra) of finite Morley rank. Let $\hk$ be a definable connected subring (respectively, $p$-subalgebra) and let $\ik$ be a definable connected ideal (respectively, $p$-ideal) contained in $\Phi(\hk)$ (respectively, in $\Phi_p(\hk)$). Then $\ik$ is contained in $\Phi(\gk)$ (respectively, in $\Phi_p(\gk)$).
\end{lem}
\begin{proof}
    Suppose for contradiction that there exists a maximal definable connected subring $\mk$ that does not contain $\ik$. Then $\gk=\ik+\mk$ and therefore $\hk=\ik+(\mk\cap\hk)=\Phi(\hk)+(\mk\cap\hk)^{\circ}$; thus, $\ik\subseteq\hk\subseteq \mk$, a contradiction.
    \\\\
    The proof for the case of a $p$-algebra is virtually identical.
\end{proof}
\begin{lem}\label{Frattini 2}(compare with \cite[Lemma 7.2]{T})
Let $\gk$ be a connected Lie ring (respectively, $p$-algebra) of finite Morley rank. Let $\ik$ be a definable connected abelian ideal (respectively, $p$-ideal) such that $(\ik\cap \Phi(\gk))^{\circ}=0$ (respectively, $(\ik\cap\Phi_p(\gk))^{\circ}=0$). Then there exists a definable connected subring (respectively, $p$-subalgebra) $\hk$ such that $\gk=\hk(+)\ik$ (we say in this case that $\gk$ is split over $\ik$).
\end{lem}
\begin{proof}
Note that there exists a definable connected subring such that $\gk=\ik+\ck$ and take $\ck$ minimal with respect to this property. First, $(\ik\cap\ck)^{\circ}\subseteq \Phi(\ck)$. Indeed, otherwise there exists a maximal definable connected subring $\mk$ in $\ck$ that does not contain $(\ik\cap\ck)^{\circ}$. But then $\ck=\mk+(\ik\cap \ck)^{\circ}$ and $\gk=\ik+\mk+(\ik\cap \ck)^{\circ}=\ik+\mk$, a contradiction with the choice of $\ck$. By abelianity of $\ik$, we deduce that $(\ik\cap \ck)^{\circ}$ is a definable connected ideal of $\gk$ and it is therefore contained in $(\Phi(\gk)\cap \ik)^{\circ}=0$ by Lemma \ref{Frattini 1}. Thus, $\gk=\ik(+)\ck$.
\\\\
The proof for the case of a $p$-algebra is virtually identical.
\end{proof}
We introduce the socle and we make the connection with the Frattini theory.
\begin{definition}
Let $\gk$ be a connected Lie ring of finite Morley rank. We define the socle, $S(\gk)$, as the subgroup generated by the definable connected $\gk$-minimal abelian ideals.
\end{definition}
\begin{lem}
Let $\gk$ be a connected Lie ring of finite Morley rank. Then $S(\gk)$ is a definable connected abelian ideal and $S(\gk)=\ik_1(+)\dots(+)\ik_n$, where $\ik_k$ is a definable connected $\gk$-minimal abelian ideal for $k\in\{1,\dots,n\}$. Moreover, if $\gk$ is a connected $p$-algebra of finite Morley rank, then $S(\gk)$ is a $p$-algebra.
\end{lem}
\begin{proof}
 By the indecomposability theorem, $S(\gk)$ is a definable connected subgroup and there exist definable $\gk$-minimal abelian ideals $\ik_1,\dots \ik_n$ such that $S(\gk)=\ik_1+\dots +\ik_n$. We may take $n$ minimal. Now, $\ik_j\cap(\ik_1+\dots+\hat{\ik}_j+\dots +\ik_n)$ is an ideal contained in $\ik_j$; $\gk$-minimality and the choice of $n$ imply that it is a finite ideal. In particular, $[\ik_j,\ik_k]\subseteq\ik_i\cap \ik_j$, and therefore $[\ik_j,\ik_i]=0$. Consequently, $S(\gk)$ is an abelian ideal.
 \\\\
 Regarding stability under the $p$-map, it suffices to note that a definable $\gk$-minimal abelian ideal is a $p$-ideal.
\end{proof}
\begin{cor}\label{scission p-Frattini}
Let $\gk$ be a connected soluble $p$-algebra of finite Morley rank such that $p>MR(\gk)$. Suppose that $\Phi_p(\gk)^{\circ}=0$. Then $\gk=S(\gk)(+)\hk$, where $\hk$ is a definable connected $p$-subalgebra.
\end{cor}
\begin{proof}
It suffices to apply Lemma \ref{Frattini 2} for $\ik=S(\gk)$.
\end{proof}
\begin{theo}\label{socle et frattini}
Let $\gk$ be a connected soluble Lie ring of finite Morley rank such that $\mathrm{char}(\gk)>MR(\gk)$. Then $\Phi(\gk)^{\circ}=0$ if and only if $\gk=S(\gk)(+)\hk$, for a definable connected subring $\hk$.
\end{theo}
\begin{proof}
Suppose that $\gk=S(\gk)(+)\hk$ for $\hk$ a definable connected subring. Suppose for contradiction that $\Phi(\gk)^{\circ}\neq 0$ and let $\ik$ be a definable connected $\gk$-minimal ideal contained in $\Phi(\gk)^{\circ}$. Now, $S(\gk)=\ik_1(+)\dots (+)\ik_r$. But \[\bk_i=\hk+(\ik_1)(+)\dots (+)\hat{\ik}_i(+)\dots (+)\ik_r\] is a maximal definable connected subring (otherwise, there would exist a maximal definable connected subring $\mk$ containing $\bk_i$ and $\gk=\mk(+)\ik_i=\bk_i+\ik_i$; thus, $MR(\mk)+MR(\ik_i)=MR(\bk_i)+MR(\ik_i)$ and therefore $MR(\mk)=MR(\bk_i)$, a contradiction). Therefore \[\Phi(\gk)\subseteq\bigcap_i \bk_i\subseteq \hk+(\ik_1\cap \dots \cap \ik_n)\] and in particular, $\ik\subseteq\Phi(\gk)^{\circ}\subseteq \hk$, a contradiction since $(\hk\cap S(\gk))^{\circ}=0$.
\\
Conversely, the result follows from Lemmas \ref{Frattini 1} and \ref{Frattini 2}.
\end{proof}

\begin{theo}
Let $\gk$ be a connected soluble $p$-algebra of finite Morley rank such that $p>MR(\gk)$. Then $\Phi(\gk)^{\circ}\subseteq \Phi_p(\gk)^{\circ}$.
\end{theo}
\begin{proof}
  First, $\Phi_p^{\circ}(\gk/\Phi^{\circ}_p(\gk))=0$; consequently, $\gk/\Phi^{\circ}_p(\gk)$ is split over $S(\gk/\Phi^{\circ}_p(\gk))$ (Corollary \ref{scission p-Frattini}). By Theorem \ref{socle et frattini}, $\Phi^{\circ}(\gk/\Phi^{\circ}_p(\gk))=0$ and therefore $\Phi(\gk)^{\circ}\subseteq\Phi_p^{\circ}(\gk)$.
\end{proof}

\section{Minimal simple $p$-algebras}
We indicate how the preceding results could contribute to a classification project for simple $p$-algebras of finite Morley rank that are ``minimal'' in a sense we will now make precise.
\\\\
First, we recall the construction of the Witt algebra $W(1,1)$. Consider the algebra of truncated polynomials $A=\K[X]/(X^p)=\K[x]$, for $\K$ a field of characteristic $p>0$; we define the Witt algebra as the algebra of derivations of $A$: this is a $p$-Lie algebra with the usual bracket for endomorphisms, where the $p$-map simply corresponds to the $p$-th power map. Now, a derivation $D$ is entirely determined by the value $D(x)$; denoting $\delta=\frac{d}{dx}$, we therefore have a natural basis given by $e_i=x^{i+1}\delta$ for $i\in \{-1,\dots, p-2\}$. In this case, the only semi-simple element of this basis is $e_0$, which therefore generates a $p$-good torus of dimension 1; moreover, $e_i^p=0$ for $1\leq i$. Note also that $e_{-1}, e_0, e_1$ generate a copy of $\mathfrak{sl}(2,K)$. The classification of subalgebras of $W(1,1)$ given in \cite{St} easily implies the following property: let $\bk$ be a non-trivial soluble subalgebra; then $N_W(\bk)$ is soluble.
\\\\
We are led to the following definition.
\begin{definition}
A $p$-algebra of finite Morley rank is said to be minimal if it is simple, if $MR(\gk)\leq p$, and if for any non-trivial definable connected soluble $p$-subalgebra $\bk$, $N^{\circ}_{\gk}(\bk)$ is soluble.
\end{definition}
From the point of view of the classification of simple Lie algebras of finite linear dimension over an algebraically closed field of positive characteristic, the Witt algebra and $\mathfrak{sl}_2$ are the only ``minimal'' $p$-algebras.
\\
Note that our minimality hypothesis allows us to use all the results of the preceding sections concerning definable soluble $p$-algebras.
\begin{conjec}
 Let $\gk$ be a minimal $p$-algebra of finite Morley rank. Then $\gk\simeq \mathfrak{sl}_2(\K)$ or $\gk\simeq W(1,1)$, for a definable field $\K$ of characteristic $p>0$.
\end{conjec}

\bibliography{p-algebre}
\bibliographystyle{plain}

\end{document}